\documentclass[a4]{amsart}

\input xypic 
\input xy 
\xyoption{all} 
\usepackage{amssymb} 
\usepackage{bbm}
\usepackage{tikz-cd}

\usepackage[bookmarksnumbered, bookmarksopen,
colorlinks,citecolor=blue,linkcolor=blue,backref]{hyperref}

\newtheorem{theorem}{Theorem}[section]
\newtheorem{proposition}[theorem]{Proposition}
\newtheorem{lemma}[theorem]{Lemma}

\theoremstyle{definition}

\newtheorem{example}[theorem]{Example}

\newcounter{bean}

\newcommand{\namedright}[3]{\ensuremath{#1\stackrel{#2}
 {\longrightarrow}#3}}
\newcommand{\nameddright}[5]{\ensuremath{#1\stackrel{#2}
 {\longrightarrow}#3\stackrel{#4}{\longrightarrow}#5}}

\newcommand{\larrow}{\relbar\!\!\relbar\!\!\rightarrow}
\newcommand{\llarrow}{\relbar\!\!\relbar\!\!\larrow}

\newcommand{\lnamedright}[3]{\ensuremath{#1\stackrel{#2}
 {\larrow}#3}}
\newcommand{\lnameddright}[5]{\ensuremath{#1\stackrel{#2}
 {\larrow}#3\stackrel{#4}{\larrow}#5}}

\newcommand{\llnameddright}[5]{\ensuremath{#1\stackrel{#2}
 {\llarrow}#3\stackrel{#4}{\llarrow}#5}}

\newcommand{\qqed}{\hfill\Box}

\begin{document}


\title{Homotopy theoretic properties of open books} 

\author{Ruizhi Huang} 
\address{Institute of Mathematics, Academy of Mathematics and Systems Science, 
   Chinese Academy of Sciences, Beijing 100190, China} 
\email{huangrz@amss.ac.cn} 
   \urladdr{https://sites.google.com/site/hrzsea/}
   
\author{Stephen Theriault}
\address{School of Mathematics, University of Southampton, Southampton 
   SO17 1BJ, United Kingdom}
\email{S.D.Theriault@soton.ac.uk}

\date{}


\begin{abstract} 
We study the homotopy groups of open books in terms of those of their pages and bindings. 
Under homotopy theoretic conditions on the monodromy we prove an integral decomposition result for the based loop space on an open book, and under more relaxed conditions prove a rational loop space decomposition. The latter case allows for a rational dichotomy theorem for open books, as an extension of the classical dichotomy in rational homotopy theory. As a direct application, we show that for Milnor's open book decomposition of an odd sphere with monodromy of finite order the induced action of the monodromy on the homology groups of its page cannot be nilpotent. 
\end{abstract}

\maketitle


\section{Introduction}

Open book decomposition is a convenient way to study manifolds. The purpose of this paper is to gain 
insight into the homotopy theory of open books by studying their based loop spaces. This fits into a larger program 
aimed at establishing homotopy theoretic properties of manifolds, and more generally, Poincar\'{e} 
Duality complexes~\cite{BB, BW, BT1,BT2,H,HT1,HT2,HT3,T}. 

Let $V$ be a smooth compact $(n-1)$-manifold with $\partial V \neq \emptyset$. Let $h$ be a self diffeomorphism of~$V$ which restricts to the identity on $\partial V$. Let $V_h$ be the {\it mapping torus} of $h$, defined as the quotient space $V_{h}=(V\times I)/\sim$ where $I=[0,1]$ is the unit interval and $(v,0)\sim (h(v),1)$. This has boundary 
$\partial V\times S^{1}$ and projection to the second coordinate induces a fibre bundle
\begin{equation}\label{Vhbundleeq}
V\stackrel{}{\longrightarrow}V_h\stackrel{\pi}{\longrightarrow} S^1.
\end{equation} 
Following the notation in~\cite{BC}, write $(\partial V\times D^{2})\cup_{{\rm id}} V_{h}$ for the union of $\partial V\times D^{2}$ and $V_{h}$ over 
the common subspace $\partial V\times S^{1}$. 
A closed $n$-manifold $M$ is an {\it open book} if there is a diffeomorphism 
\begin{equation}\label{openbookdefeq}
M\cong (\partial V\times D^2)\cup_{\rm id} V_h,
\end{equation}
for some $V$ and $h$ as above. The map $h$ is called the {\it monodromy} of the open book.

Open books are of important interest in both topology and geometry. For instance, the fundamental open book theorem of Winkelnkemper \cite{W} states that a simply-connected manifold of dimension greater than $6$ is an open book if and only if its signature is $0$. Classical work of Milnor~\cite{Mil} provides explicit open book decompositions of odd dimensional spheres. 
Other historical applications of open books can be found in \cite[Appendix]{Ran}. More recently, Gitler and Lopez de Medrano~\cite{GL} used open books as a tool to show that certain families of manifolds arising from combinatorial constructions are diffeomorphic to connected sums of products of spheres. In the context of contact geometry, a remarkable work by Giroux \cite{Gi} showed that a contact manifold admits an open book structure that is compatible with the contact structure. Very recently, Bowden and Crowley \cite{BC} gave a topological obstruction to the existence of an open book structure on a contact manifold that has flexible pages. 

There are other equivalent descriptions of open books.
Up to homeomorphism, the open book~$M$ can be obtained from $V_h$ by identifying $(x, t)$ with $(x, s)$ for each $x\in \partial V$ and $t$, $s\in S^1$.
For $t\in S^1$ the fibre $\pi^{-1}(t)\cong V\times {t}$ of $V_h$ is a codimension $1$ submanifold of $M$ whose image is the $t$-th {\it page} of the open book. And the image of $\partial V\times S^1$ in $M$ is a closed codimension $2$ submanifold called the {\it binding} of the open book.
This description is the reason behind the name ``open book".

Our first result establishes an integral homotopy decomposition for the based loops on open books 
given a condition on the diffeomorphism $h$.  Recall that the loop space $\Omega X$ of a based 
topological space $X$ is the space of all pointed, continuous maps from the circle into $X$, while the 
suspension~$\Sigma X$ is the double cone of $X$. 

\begin{theorem}\label{zMdecthm}
Let $M$ be a path-connected open book for which there is a diffeomorphism 
$M\cong (\partial V\times D^2)\cup_{\rm id} V_h$. Suppose that $h\simeq {\rm id}$ relative to $\partial V$. Then there is a homotopy equivalence
\[
\Omega M\simeq \Omega V\times \Omega \Sigma^2 F 
\]
where the space $F$ is the homotopy fibre of the inclusion $\partial V\stackrel{}{\rightarrow} V$. Consequently, there is an isomorphism 
\[
\pi_\ast(M)\cong \pi_\ast(V)\oplus \pi_\ast(\Sigma^2 F).
\]
\end{theorem}

Theorem \ref{zMdecthm} is proved as a consequence of the following purely homotopy 
theoretic result, which is interesting in its own right and can be applied elsewhere. Recall that $X\wedge Y$ 
is the smash product of two based spaces $X$ and $Y$.  

\begin{theorem} 
   \label{productpushout} 
   Let $A$, $B$, $C$ and $D$ be path-connected spaces and suppose that there is a homotopy pushout 
   \[\diagram 
        A\times B\rto^-{1\times g}\dto^{f\times 1} & A\times D\dto \\ 
        C\times B\rto & Q 
     \enddiagram\] 
   that defines the space $Q$. Then there is a homotopy equivalence 
   \[\Omega Q\simeq\Omega C\times\Omega D\times\Omega \Sigma (F\wedge G)\] 
   where $F$ and $G$ are the homotopy fibres of $f$ and $g$ respectively. 
\end{theorem} 

To deal with open books where the diffeomorphism $h$ need not be homotopic to the identity 
map relative to $\partial V$ we turn to rational homotopy theory. The argument will involve 
another equivalent description of open books; see for instance \cite{Q, BC}. Let 
\[DV=\partial(V\times I)\] 
be the {\it (trivial) double} of $V$ and let 
\(i\colon\namedright{DV}{}{V\times I}\) 
be the inclusion of the boundary. 
Define a self-diffeomorphism $e(h)\colon DV\longrightarrow DV$ that extends~$h$ by  
$e(h)(v,0)=v$, $e(h)(v,t)=(v,t)$ if $v\in\partial V$ and $e(h)(v,1)=h(v)$. Let 
$(V\times I)\cup _{e(h)} (V\times I)$ be the manifold obtained by gluing together the image of $i$ in the left 
copy of $V\times I$ and the image of $e(h)\circ i$ in the right copy of $V\times I$. 
The open book~$M$ in~(\ref{openbookdefeq}) can be expressed as a {\it twisted double} via a diffeomorphism 
\begin{equation}\label{openbookt2def}
M\cong (V\times I)\cup_{e(h)}(V\times I). 
\end{equation}  

Recall that a path-connnected space $X$ is called {\it nilpotent} if its fundamental group $\pi_1(X)$ is a nilpotent group and it acts nilpotently on the higher homotopy groups $\pi_i(X)$ for $i\geq 2$. Simply-connected spaces, connected $H$-spaces and loop spaces are examples of nilpotent spaces. 

\begin{theorem}\label{Mdecthm}
Let $M$ be a path-connected open book such that $M\cong (V\times I)\cup_{e(h)}(V\times I)$, 
where~$V$ and $\partial V$ are path-connected and nilpotent, and the inclusion of the boundary 
\(\namedright{\partial V}{}{V}\) 
is $1$-connected. Suppose that $i_{\ast}\circ e(h)_\ast: \pi_\ast(DV)\otimes \mathbb{Q}\rightarrow\pi_\ast(DV)\otimes \mathbb{Q}$ equals $i_{\ast}$. Then there is a rational homotopy equivalence
\[
\Omega M\simeq_{\mathbb{Q}} \Omega V\times \Omega \Sigma^2 F,
\]
where the space $F$ is the homotopy fibre of the inclusion $\partial V\stackrel{}{\rightarrow} V$.
\end{theorem}
The decomposition in Theorem \ref{Mdecthm} is a rational version of that in Theorem \ref{zMdecthm} under a looser condition on the monodromy $h$. In Theorem \ref{Mdecthm}, the nilpotence conditions 
on $V$ and~$\partial V$ are mild. 
Any simply-connected space is nilpotent, and for a simply-connected open book of dimension greater than $6$, Winkelnkemper~\cite{W} showed that both $V$ and $\partial V$ can be chosen to be simply-connected. 

As an application of Theorem~\ref{Mdecthm} we study the relationship between the rational homotopy groups of manifolds and their possible open book structures. In rational homotopy theory, there is a classical dichotomy characterizing rational spaces~\cite[page 452]{FHT}, \cite[Section 2.5.3]{FOT}.

\textit{Any connected nilpotent space $X$ with rational homology of finite type and finite rational Lusternik-Schnirelmann category is either:
\begin{itemize}
\item rationally elliptic, that is, $\pi_\ast(X)\otimes \mathbb{Q}$ is finite dimensional, or else
\item rationally hyperbolic, that is, $\pi_\ast(X)\otimes \mathbb{Q}$ grows exponentially.
\end{itemize}
}
\noindent 
A connected finite dimensional $CW$-complex, for example, has finite Lusternik-Schnirelmann 
category. Since a smooth compact manifold can be given the structure of a finite dimensional $CW$-complex, it has rational homology of finite type and finite rational Lusternik-Schnirelmann category.

\begin{theorem}
\label{openbookthm}
Let $M$ be a path-connected $n$-manifold satisfying an open book decomposition $M\cong (\partial V\times D^2)\cup_{\rm id} V_h$, where $V$ and $\partial V$ are path-connected and nilpotent, and the inclusion 
of the boundary  
\(\namedright{\partial V}{}{V}\) 
induces an epimorphism of fundamental groups. Suppose that one of the following holds: 
\begin{itemize}
\item[(a)] the induced homomorphism 
$i_{\ast}\circ e(h)_\ast: \pi_\ast(DV)\otimes \mathbb{Q}\rightarrow \pi_\ast(V\times I)\otimes \mathbb{Q}$ 
equals $i_{\ast}$;
\item[(b)] $i_{\ast}\circ e(h)_\ast^m= i_{\ast}$ for some $m\in \mathbb{Z}^{+}$ and the monodromy $h$ acts nilpotently on the homotopy groups $\pi_{\ast}(V)$.
\end{itemize}
Then either:
\begin{itemize}
\item[(1)] $M$ is rationally elliptic, in which case $V$ is also rationally elliptic, the homotopy fibre of 
the inclusion $\partial V\hookrightarrow V$ is rationally homotopy equivalent to a sphere $S^{l}$, and 
\[
\pi_\ast(M)\otimes \mathbb{Q}\cong (\pi_\ast(V)\otimes \mathbb{Q})\oplus (\pi_\ast(S^{l+2})\otimes \mathbb{Q}); 
\]
\item[(2)] $M$ is rationally hyperbolic, in which case either $V$ is rationally hyperbolic or the homotopy fibre of the inclusion $\partial V\hookrightarrow V$ is not rationally homotopy equivalent to a sphere.
\end{itemize}
\end{theorem} 

In \cite{BC}, Bowden and Crowley proved that if $M$ is a contact manifold admitting an open book structure whose pages are flexible Weinstein manifolds, then the map $e(h): DV\stackrel{e(h)}{\longrightarrow} DV$ has the property that $e(h)_\ast: H_\ast(DV;\mathbb{Z})\rightarrow H_\ast(DV;\mathbb{Z})$ is the identity map up to half the dimension of $V$. 
In contrast, our conditions in parts~(a) and~(b) of Theorem~\ref{openbookthm} for $e(h)_\ast$ on the rational homotopy groups can be viewed as a sort of {\it homotopy order condition} on the monodromy. The condition in part~(a) is satisfied, for example, if the monodromy $h$ is homotopic to the identity and the condition in part~(b) is satisfied, for example, if $h$ is of finite order. Viewed this way, 
Theorem~\ref{openbookthm} gives a rational dichotomy of open books with a homotopy order condition. Having one of the conditions in parts~(a) and~(b) hold is necessary, this is illustrated in Proposition~\ref{milcoro} using Milnor's open book decompositions of odd spheres. Further, the nilpotent action condition for the monodromy in case (b) is necessary as illustrated in Example~\ref{Keg}. In Section \ref{sec:necessity}, we also compare Theorem \ref{openbookthm} with a result of Grove and Halperin \cite{GH} on the rational ellipticity of double mapping cylinders.

It is worth pointing out that, as methods are homotopy theoretical, the salient points in 
the arguments are not that $V$ is a manifold and $h$ is a diffeomorphism but that the inclusion 
\(\namedright{\partial V}{}{V}\) 
is not a rational equivalence and $h$ is a homotopy equivalence.

The paper is organized as follows. 
In Section \ref{sec: 0tdoub} we prove the general result Theorem \ref{productpushout}, give an integral loop space decomposition for open books with monodromy homotopic to the identity map, and prove Theorem \ref{zMdecthm}. 
In Section~\ref{sec: doub} we give a loop space decomposition of the double of a manifold with boundary.
We then turn to the rational homotopy of open books. In Section \ref{sec: tdoub} we consider a special case of the homotopy order condition for open books, give a rational loop space decomposition for such open books, and prove Theorem \ref{Mdecthm}.  
Section \ref{sec: proofthm1} is devoted to proving the dichotomy result, Theorem \ref{openbookthm}, and Section~\ref{sec:necessity} gives an example proving the necessity of the homotopy order condition in Theorem~\ref{openbookthm}.
Since our techniques in homotopy theory are based on the two cube theorems of Mather \cite{M}, we state and comment on them in Appendix \hyperref[AppendixA]{A} for those interested readers from manifold topology and geometry.

\bigskip

\noindent{\bf Acknowledgements.}
The first author was supported in part by the National Natural Science Foundation of China (Grant nos. 12331003 and 12288201), the National Key R\&D Program of China (No. 2021YFA1002300), the Youth Innovation Promotion Association of Chinese Academy Sciences, and the ``Chen Jingrun'' Future Star Program of AMSS. 
The authors would like to thank Xiaoyang Chen for suggesting a study of the rational ellipticity of open books and for related discussions, and Zhengyi Zhou for helpful discussions on Milnor fibrations. 

The authors would like to thank one referee for pointing out the necessity of the nilpotence 
condition in Theorem~\ref{openbookthm} and for Example~\ref{Keg}, and are indebted to another referee for 
many valuable comments, including suggesting Theorem~\ref{productpushout} as a means to prove 
Theorem~\ref{zMdecthm}.

\section{An integral loop space decomposition of open books with homotopically trivial monodromy}
\label{sec: 0tdoub}

As we will need to work with homotopy fibrations and homotopy groups, 
throughout it will be assumed that all spaces and maps are pointed. In particular, if $V$ is a compact 
manifold with boundary $\partial V\neq \emptyset$ then assume that a basepoint $v_{0}$ for $V$ 
has been chosen that is also in $\partial V$. 

We begin by proving the general result, Theorem \ref{productpushout}. Let $I=[0,1]$ be the unit interval with $0$ as basepoint. 
For path-connected spaces $X$ and $Y$, the (reduced) \emph{join} is the quotient space 
\[X\ast Y=(X\times I\times Y)/\sim\] 
where $(x,0,y)\sim (x',0,y)$, $(x,1,y)\sim (x,1,y')$ and $(\ast,t,\ast)\sim (\ast,0,\ast)$ for all  
$x,x'\in X$, $y,y'\in Y$ and $t\in I$. It is well known that there is a homotopy equivalence 
$X\ast Y\simeq\Sigma X\wedge Y$. 

\begin{proof}[Proof of Theorem \ref{productpushout}] 
First observe that there is a pushout map 
\begin{equation} 
\begin{aligned}
  \label{fgpo} 
  \xymatrix{ 
     A\times B\ar[r]^{1\times g}\ar[d]^{f\times 1} & A\times D\ar@/^/[ddr]^{f\times 1}\ar[d] & \\ 
     C\times B\ar[r]\ar@/_/[drr]_{1\times g} 
         & Q\ar@{.>}[dr]^(0.4){\theta} & \\ 
     & & C\times D }
     \end{aligned}
\end{equation}  
for some map $\theta$. Let $H$ be the homotopy fibre of $\theta$. Pulling back 
\(\namedright{H}{}{Q}\) 
with each of the maps in the homotopy pushout defining $Q$ then gives a homotopy 
commutative cube  
\begin{gather}
\begin{aligned}
\xymatrixcolsep{1.5pc}
\xymatrixrowsep{1.5pc}
\xymatrix{
 & F\times G \ar[dl]_{b}  \ar[rr]^{a}  \ar[dd]|!{[d];[d]}\hole && 
 F \ar[dl]  \ar[dd]^{}  \\
 G \ar[dd]^{}  \ar[rr]  &&
H \ar[dd] \\
  & A\times B \ar[dl]_{f\times 1}  \ar[rr]^<<<<{1\times g}|!{[r];[r]}\hole  && 
 A\times D \ar[dl]    \\
C\times B   \ar[rr]  &&
Q
}
\end{aligned}
\label{generalfgcube}
\end{gather} 
in which the bottom face is a homotopy pushout, the four sides are homotopy pullbacks, and the 
maps $a$ and $b$ are induced maps of fibres. Mather's second cube theorem (Theorem \ref{2cubthm}) 
implies that the top face is a homotopy pushout. 

We now identify the homotopy classes of $a$ and $b$. The rear face of the cube is the left square in the 
homotopy fibration diagram 
\[\diagram 
     F\times G\rto\dto^{a} & A\times B\rto^-{f\times g}\dto^{1\times g} & C\times D\ddouble \\ 
     F\times\ast\rto & A\times D\rto^-{f\times 1} & C\times D  
  \enddiagram\] 
where the $\ast$ in the lower left corner has been added to make clear the lower row 
is a product of homotopy fibrations. Observe that the entire diagram is a product of two 
homotopy fibration diagrams, one for the left factors and one for the right factors. 
This implies that $a$ is the product $1\times\ast$, that is, $a$ is homotopic to the projection 
\(\namedright{F\times G}{\pi_{1}}{F}\) 
onto the first factor. Similarly, $b$ is homotopic to the projection 
\(\namedright{F\times G}{\pi_{2}}{G}\) 
onto the second factor. 

The homotopy pushout in the top face of~(\ref{generalfgcube}) therefore implies that $H$ is 
homotopy equivalent to the pushout of the projections $\pi_{1}$ and $\pi_{2}$, which is 
homotopy equivalent to $F\ast G$. Therefore, we obtain a homotopy fibration 
\(\nameddright{F\ast G}{}{Q}{\theta}{C\times D}\). 
The pushout diagram~(\ref{fgpo}) defining $\theta$ implies that both $C$ and $D$ map to $Q$ 
and give a composite 
\(\nameddright{C\vee D}{}{Q}{\theta}{C\times D}\)  
that is homotopic to the inclusion of the wedge into the product. By~\cite{Ga}, this inclusion 
has a right homotopy inverse after looping, implying that $\Omega\theta$ has a right 
homotopy inverse. Thus the homotopy fibration 
\(\nameddright{F\ast G}{}{Q}{\theta}{C\times D}\) 
splits after looping to give a homotopy equivalence 
$\Omega Q\simeq\Omega C\times\Omega D\times\Omega (F\ast G)$. 
\end{proof} 

To apply Theorem~\ref{productpushout} in the context of open books a lemma is first required. Let 
\(\iota\colon\namedright{\partial V}{}{V}\) 
be the inclusion of the boundary. 

\begin{lemma} 
   \label{Vsplit} 
   Let $V$ be a smooth compact $(n-1)$-manifold with nonempty boundary and $h$ a self-diffeomorphism 
   of $V$ that restricts to the identity on $\partial V$. If $h\simeq {\rm id}$ relative to $\partial V$ then there is a  
   homotopy equivalence $V_{h}\simeq V\times S^{1}$ satisfying a commutative diagram 
   \[\diagram 
         \partial V_{h}\rto^-{=}\dto & \partial V\times S^{1}\dto^{\iota\times 1} \\ 
         V_{h}\rto^-{\simeq} & V\times S^{1}. 
     \enddiagram\] 
\end{lemma} 

\begin{proof} 
Recall that $V_{h}=(V\times I)/\sim$ where $(v,0)\sim (h(v),1)$ and there is a fibre bundle 
\(\nameddright{V}{j}{V_{h}}{\pi}{S^{1}}\) 
where $j(v)=[v,0]$ and $\pi([v,t])=e^{2\pi it}$ is the projection to the second factor. 
Since ${\rm id}\simeq h^{-1}$ relative 
to $\partial V$, there is a homotopy 
\(H\colon\namedright{V\times I}{}{V}\) 
such that $H_{0}={\rm id}$, $H_{1}=h^{-1}$ and the restriction of $H$ to $\partial V\times I$ is the projection 
onto $\partial V$. 
Define $\overline{H}: V\times I \longrightarrow V\times S^1$ by $\overline{H}(v, t)=(H(v,t), q(t))$ with $q(t)=e^{2\pi it}$. Since $\overline{H}(h(v), 1)=(h^{-1}(h(v)), 1)=(v,1)=\overline{H}(v, 0)$, the map $\overline{H}$ reduces to a map 
\[
G: V_h\longrightarrow V\times S^1
\]
such that $G([v, t])=(H(v,t), q(t))$. In particular, $(\pi_2\circ G)([v, t])=q(t)=\pi([v,t])$, where $\pi_2: V\times S^1 \longrightarrow S^1 $ is the projection, that is $\pi_2\circ G=\pi_2$. 
For any fixed $q(t)\in S^1$, the restriction on fibres $G: \pi^{-1}(q(t))\stackrel{}{\longrightarrow} \pi_2^{-1}(q(t))=V\times \{q(t)\}$ is a homotopy equivalence as $H_t$ is. Hence~$G$ 
induces a fibrewise homotopy equivalence between the bundle 
\(\nameddright{V}{j}{V_{h}}{\pi}{S^{1}}\) 
and the trivial bundle 
\(\nameddright{V}{}{V\times S^{1}}{\pi_{2}}{S^{1}}\). 
Since the restriction of $H$ to $\partial V\times I$ is the projection 
onto $\partial V$, the restriction of 
this fibrewise homotopy equivalence $G$ to $\partial V_{h}$ is the inclusion 
\(\namedright{\partial V\times S^{1}}{\iota\times 1}{V\times S^{1}}\), 
giving the asserted commutative diagram. 
\end{proof} 

We can now prove a homotopy decomposition for the based loops on a family of open books. 

\begin{proof}[Proof of Theorem~\ref{zMdecthm}] 
Recall from~(\ref{openbookdefeq}) that $M\cong (\partial V\times D^{2})\cup_{\partial V\times S^{1}} V_{h}$. 
By hypothesis, $h\simeq\rm{id}$ relative to $\partial V$, so by 
Lemma~(\ref{Vsplit}) the space $V_{h}$ may be replaced up to homotopy equivalence with 
$V\times S^{1}$ and in a way that is compatible with the inclusion of the boundary 
$\partial V_{h}=\partial V\times S^{1}$. Thus~$M$ is homotopy equivalent to 
$(\partial V\times D^{2})\cup_{\partial V\times S^{1}} (V\times S^{1})$, that is, there is a homotopy pushout 
\begin{equation} 
  \label{openbookproductpushout} 
  \diagram 
    \partial V\times S^{1}\rto^-{1\times j}\dto^{\iota\times 1} & \partial V\times D^{2}\dto  \\ 
    V\times S^{1}\rto & M 
  \enddiagram 
\end{equation}  
where $j$ is the standard inclusion. By Theorem~\ref{productpushout}, there is a homotopy equivalence 
\[\Omega M\simeq\Omega V\times\Omega D^{2}\times\Omega (F\ast G)\] 
where $F$ and $G$ are the homotopy fibres of $\iota$ and $j$ respectively. As $D^{2}$ is contractible 
we obtain $G\simeq S^{1}$ and therefore $F\ast G\simeq\Sigma^{2} F$. Thus 
$\Omega M\simeq\Omega V\times\Omega\Sigma^{2} F$. 
\end{proof}

\section{A loop space decomposition of the double of $V$}
\label{sec: doub}
By definition, the double of $V$ is $DV=\partial(V\times I)$. Define the folding map 
\[p: DV\rightarrow V\] 
by $p(v,0)=v$, $p(v,t)=v$ for $v\in\partial V$ and $p(v,1)=v$. Let $P$ be the homotopy fibre of~$p$.

\begin{lemma}\label{DVdeclemma}
There are homotopy equivalences
\[
\Omega DV\simeq \Omega V\times \Omega \Sigma F \ \ \ ~{\rm and} \ \ \ ~P\simeq \Sigma F.
\]
\end{lemma}
\begin{proof} 
Collapsing the cylinder $\partial V\times I$ in $DV=\partial(V\times I)=(V\times\{0\})\cup(\partial V\times I)\times (V\times\{1\})$ 
implies that there is a pushout, up to homotopy equivalences, 
 \[
\begin{gathered}
\xymatrix{
\partial V \ar[r]^{\iota} \ar[d]^{\iota} &
V\ar[d]^{\jmath_2} \\
V\ar[r]^{\jmath_1} &
DV
}
\end{gathered}
\] 
where the maps $\jmath_i$ for $i=1$, $2$ are the inclusions into the top and bottom copies of $V$. 

Compose each of the four corners  of the pushout with the folding map $p$ and take homotopy fibres. Noting that $p\circ\jmath_{1}$ and $p\circ\jmath_{2}$ are both the identity map on~$V$, we obtain homotopy fibrations
\[
\begin{split}
   P\stackrel{\widetilde{f}}{\longrightarrow} DV\stackrel{p}{\longrightarrow} V \\ 
   \ast \stackrel{}{\longrightarrow}  V \stackrel{p\circ\jmath_{1}}{\longrightarrow}  V \\
   \ast \stackrel{}{\longrightarrow}  V \stackrel{p\circ\jmath_{2}}{\longrightarrow}  V \\ 
    F\stackrel{f}{\longrightarrow}  \partial V \stackrel{\iota}{\longrightarrow}  V 
\end{split}
\] 
that define the spaces $P$ and $F$ and the maps $\widetilde{f}$ and $f$. 
In 
each of these homotopy fibrations the map from the total space to the base factors through 
the map 
\(DV\stackrel{p}{\longrightarrow}{V}\), 
so we obtain a homotopy commutative cube 
\begin{gather}
\begin{aligned}
\xymatrixcolsep{1.5pc}
\xymatrixrowsep{1.5pc}
\xymatrix{
 & F \ar[dl]  \ar[rr]  \ar[dd]^<<<<<{f}|!{[d];[d]}\hole && 
 \ast \ar[dl]  \ar[dd]^{}  \\
 \ast \ar[dd]^{}  \ar[rr]  &&
P \ar[dd]^<<<<<{\widetilde{f}} \\
  & \partial V \ar[dl]_{\iota}  \ar[rr]^<<<<<{\iota}|!{[r];[r]}\hole  && 
 V \ar[dl]_{\jmath_2}    \\
V   \ar[rr]^{\jmath_1}  &&
DV.
}
\end{aligned}
\label{cubeFPdiag}
\end{gather} 
in which the bottom face is a homotopy pushout and the four sides are homotopy pullbacks. Mather's second cube theorem (Theorem \ref{2cubthm}) implies that the top face is a homotopy pushout.
In particular, the top face of the cube being a homotopy pushout implies that $P\simeq \Sigma F$ while the right face of the cube being a homotopy pullback implies that there is a homotopy fibration diagram 
 \[
\begin{gathered}
\xymatrix{
\Omega V \ar@{=}[d] \ar[r]  & 
\ast \ar[d] \ar[r]^{} &
V \ar[d]^{\jmath_1} \\
\Omega V \ar[r]^{\delta} &
P \ar[r]^{\widetilde{f}}&
DV.
}
\end{gathered}
\] 
The right square implies that the connecting map $\delta$ is null homotopic. It follows that there is a homotopy equivalence 
\[
\Omega DV \simeq \Omega V \times \Omega P\simeq \Omega V \times \Omega \Sigma F,
\]
proving the lemma.
\end{proof}

\section{A rational loop space decomposition of certain open books}
\label{sec: tdoub} 
By (\ref{openbookt2def}) there is a pushout diagram
 \begin{equation}
 \label{Mpushdiag}
\begin{gathered}
\xymatrix{
DV  \ar[r]^{i\circ e(h)}  \ar[d]^{i}&
V\times I\ar[d]^{j_2} \\
V\times I\ar[r]^{j_1} &
M,
}
\end{gathered}
\end{equation}
where $i$ is the inclusion 
\(DV=\partial(V\times I)\hookrightarrow V\times I\), 
and $j_1$ and $j_{2}$ are the induced injections.
Since $i$ is a cofibration, the pushout (\ref{Mpushdiag}) is also a homotopy pushout. 

In this section, we study the loop space homotopy type of $M$ via (\ref{Mpushdiag}) with a looser condition on the monodromy. To do so we have to pass to rational homotopy. Some lemmas are needed to prepare the way. 
Recall that $P$ is the homotopy fibre of the folding map $p: DV\rightarrow V$. 

\begin{lemma}\label{fibrei=plemma}
The homotopy fibre of the inclusion $i: DV\rightarrow V\times I$ is homotopy equivalent to $P$.
\end{lemma}
\begin{proof} 
The projection 
\(\namedright{V\times I}{\pi}{V}\) 
is a homotopy equivalence, so the homotopy fibre of $i$ is homotopy equivalent to the homotopy 
fibre of $\pi\circ i$. But $\pi\circ i=p$ and, by definition, the homotopy fibre of $p$ is $P$. 
\end{proof}

Recall that $F$ is the homotopy fibre of the inclusion of the boundary $\iota: \partial V\stackrel{}{\rightarrow} V$. A map 
\(\namedright{X}{f}{Y}\)  
is $m$-connected if $f$ induces an isomorphism on $\pi_{k}$ for $k<m$ and an epimorphism on $\pi_{m}$.

\begin{lemma}\label{FPconlemma}
If $V$ and $\partial V$ are path-connected and nilpotent, and the inclusion 
\(\namedright{\partial V}{}{V}\) 
is $1$-connected, then~$F$ is path-connected, $P\simeq \Sigma F$ is simply-connected, and both $DV$ 
and $M$ are path-connected and nilpotent. 
\end{lemma}

\begin{proof}
Since 
\(\namedright{\partial V}{}{V}\) 
is $1$-connected, it induces an isomorphism on $\pi_{0}$ and an epimorphism on~$\pi_{1}$. 
The long exact sequence of homotopy groups for the homotopy fibration 
$F\stackrel{}{\longrightarrow}\partial V\stackrel{\iota}{\longrightarrow} V$ 
then implies that $F$ is path-connected. By Lemma \ref{DVdeclemma}, $P\simeq \Sigma F$, 
implying that~$P$ is simply-connected. 

The homotopy fibration 
$P\rightarrow DV\stackrel{p}{\rightarrow} V$ 
then implies that $DV$ is path-connected. As in the proof of Lemma~\ref{DVdeclemma}, there 
is a homotopy pushout 
\[\diagram 
     \partial V\rto^-{\iota}\dto^{\iota} & V\dto \\ 
     V\rto & DV. 
  \enddiagram\] 
In general, Rao~\cite[Theorem 2.1]{Rao} shows that if $W$ is the homotopy pushout of maps 
\(\namedright{X}{f}{Y}\) 
and 
\(\namedright{X}{g}{Z}\) 
where $X$, $Y$ and $Z$ are all nilpotent and $f$ and $g$ induce epimorphisms in $\pi_{1}$, 
then $W$ is nilpotent. In our case, both $\partial V$ and $V$ are path-connected and nilpotent 
and $\iota$ induces an epimorphism on $\pi_{1}$, so $DV$ is nilpotent. 

By Lemma \ref{fibrei=plemma} there is a homotopy fibration 
$P\stackrel{}{\longrightarrow}DV\stackrel{i}{\longrightarrow} V\times I$. 
As $P\simeq \Sigma F$ is simply-connected, the long exact 
sequence of homotopy groups for the fibration implies that $DV\stackrel{i}{\longrightarrow} V\times I$ 
is $2$-connected. In the pushout (\ref{Mpushdiag}), $V\times I$ is path-connected and nilpotent 
by hypothesis, we have seen that $DV$ is path-connected and nilpotent, and as $i$ is $2$-connected 
both $i$ and $i\circ e(h)$ induce epimorphisms in $\pi_1$. Therefore, by~\cite[Theorem 2.1]{Rao}, 
the pushout $M$ is path-connected and nilpotent. 
\end{proof}

We can now prove Theorem \ref{Mdecthm}.

\begin{proof}[Proof of Theorem \ref{Mdecthm}] 
The plan is to construct a rational homotopy fibration 
\(\nameddright{\Omega V}{\tau}{\Sigma P}{}{M}\) 
where $\tau$ is null homotopic, and use Lemma~\ref{FPconlemma} to identify $P$ as $\Sigma F$. 
To do this the first cube theorem (Theorem~\ref{1cubthm}) will be used to produce a  map 
\(\namedright{\Sigma P}{}{M}\) 
with the right homotopy fibre. 

To begin, we compare the homotopy fibres of the maps 
\(\namedright{DV}{i}{V\times I}\) 
and 
\(\lnamedright{DV}{i\circ e(h)}{V\times I}\). 
By Lemma~\ref{fibrei=plemma}, there is a homotopy fibration 
\(\nameddright{P}{\widetilde{g}}{DV}{i}{V\times I}\) 
where $\widetilde{g}$ is a name for the map from the fibre to the total space. Let 
$P^\prime$ is the homotopy fibre of $i\circ e(h): DV\rightarrow V\times I$. Then there is 
a homotopy fibration diagram
\begin{equation}
\label{DVehdiag}
\begin{gathered}
\xymatrix{
P^\prime \ar[r]^{\widetilde{g}^\prime}  \ar[d]^{e^\prime} &
DV\ar[d]^{e(h)} \ar[r]^{i\circ e(h)} &
V\times I \ar@{=}[d]\\
P\ar[r]^{\widetilde{g}} &
DV\ar[r]^{i} &
V\times I,
}
\end{gathered}
\end{equation}
where $e^\prime$ is an induced map of fibres. Since $e(h)$ is a diffeomorphism it induces an isomorphism of homotopy groups. The map of long exact sequences of homotopy groups induced by the map of fibrations~(\ref{DVehdiag}) and the five lemma therefore imply that $e^\prime_\ast: \pi_\ast(P^\prime)\rightarrow \pi_\ast(P)$ is an isomorphism. 
In particular, as $P$ is simply connected by Lemma~\ref{FPconlemma}, $P^\prime$ is simply-connected. 

Since $V$ is nilpotent by hypothesis, $P$ and $P'$ are simply-connected, and $DV$ is nilpotent by Lemma~\ref{FPconlemma}, we can consider 
the rationalization of Diagram (\ref{DVehdiag}).
Note the isomorphism $e^\prime_\ast$ also implies that for each $i\in \mathbb{Z}^+$,
\begin{equation}\label{dimP=P'eq}
{\rm dim}(\pi_i(P^\prime)\otimes \mathbb{Q})={\rm dim}(\pi_i(P)\otimes \mathbb{Q}).
\end{equation}
Moreover, since $V$ is compact it is of finite type, and therefore so are $DV$ and $V\times I$. This implies that both sides of (\ref{dimP=P'eq}) are finite. 

By assumption, $i_{\ast}\circ e(h)_\ast=i_{\ast}$ on rational homotopy groups. Therefore the composite 
\begin{equation}\label{iehgeq}
\pi_\ast(P)\otimes \mathbb{Q}\stackrel{\widetilde{g}_\ast}{\llarrow} \pi_\ast(DV)\otimes \mathbb{Q}  \stackrel{(i\circ e(h))_{\ast}}{\llarrow} \pi_\ast(V\times I)\otimes \mathbb{Q}
\end{equation}
is trivial as $(i\circ e(h))_{\ast}\circ \widetilde{g}_\ast=i_\ast \circ \widetilde{g}_\ast$ and $i\circ \widetilde{g}$ is null homotopic.
By Lemma \ref{FPconlemma}, $P\simeq \Sigma F$, so $P$ is a wedge of simply-connected spheres rationally.  
Thus the triviality of (\ref{iehgeq}) implies that the composite $(i\circ e(h))\circ \widetilde{g}$ is rationally null homotopic. Hence, rationally, there is a lift $e''$ of $\widetilde{g}$ to the homotopy fibre of $i\circ e(h)$, giving a homotopy commutative diagram 
\begin{equation}
\label{QPP'diag}
\xymatrixcolsep{3pc}
\xymatrixrowsep{2.5pc}
\begin{gathered}
\xymatrix{
& P \ar[d]^{\widetilde{g}} \ar[dl]_{e^{\prime\prime}} \\ 
P^\prime\ar[r]^{\widetilde{g}^\prime} &
DV. 
}
\end{gathered}
\end{equation} 

Consider again the homotopy fibration 
\(\nameddright{P}{\widetilde{g}}{DV}{i}{V\times I}\). 
The map $i$ has a right inverse up to homotopy. This is because the composite 
\(\nameddright{V}{\jmath_{1}}{DV}{i}{V\times I}\) 
is the inclusion into the base of the cylinder, where $\jmath_{1}$ is the inclusion into the base of $DV$, and this composite of inclusions is a homotopy equivalence since $I$ is contractible. Thus $\widetilde{g}_\ast$ is injective on homotopy groups. The homotopy commutativity of~(\ref{QPP'diag}) therefore implies that $e''$ is injective on rational homotopy groups. The equality in~(\ref{dimP=P'eq}) and the fact that both sides of that equality are finite then imply that $e^{\prime\prime}_\ast$ is an isomorphism on rational homotopy groups. It follows by Whitehead's Theorem that $e^{\prime\prime}$ is a rational homotopy equivalence. Hence, rationally, there is a homotopy fibration 
\(\llnameddright{P}{\widetilde{g}}{DV}{i\circ e(h)}{V\times I}\).

The nilpotence conditions in Lemma \ref{FPconlemma} imply that we may rationalize all spaces and maps to consider the cube 
\begin{gather}
\begin{aligned}
\xymatrixcolsep{1.5pc}
\xymatrixrowsep{1.5pc}
\xymatrix{
 & P \ar[dl]  \ar[rr]  \ar[dd]^<<<<{\widetilde{g}}|!{[d];[d]}\hole && 
 CP \ar[dl]  \ar[dd]^{}  \\
 CP \ar[dd]^{}  \ar[rr]  &&
Q \ar[dd]^<<<<{\widehat{g}} \\
  & DV \ar[dl]_{i}  \ar[rr]^<<<<{i\circ e(h)}|!{[r];[r]}\hole  && 
 V\times I \ar[dl]_{j_2}    \\
V\times I   \ar[rr]^{j_1}  &&
M 
}
\end{aligned}
\label{cubePdiag}
\end{gather}
where $CP$ is the rationalization of the (reduced) cone on $P$ and both $Q$ and $\widehat{g}$ will be defined 
momentarily. Since there are rational homotopy fibrations 
\(\nameddright{P}{\widetilde{g}}{DV}{i}{V\times I}\) 
and 
\(\lnameddright{P}{\widetilde{g}}{DV}{i\circ e(h)}{V\times I}\), 
the composites $i\circ\widetilde{g}$ and $(i\circ e(h))\circ\widetilde{g}$ are null homotopic. 
The (pointed) homotopy 
\(\namedright{P\times I}{}{V\times I}\) 
which at time $0$ is $i\circ\widetilde{g}$ and time $1$ is the constant map 
implies that there is a map 
\(\namedright{CP}{}{V\times I}\) 
which makes the left face of~(\ref{cubePdiag}) strictly commute. Similarly, there is a map 
\(\namedright{CP}{}{V\times I}\) 
which makes the rear face of~(\ref{cubePdiag}) strictly commute. Note the two maps from $CP$ 
to $V\times I$ may be different if the homotopies are different. 
The bottom face of~(\ref{cubePdiag}) will strictly commute once we replace $M$ by the point-set pushout of the rational maps $i$ and $i\circ e(h)$, which is rationally homotopy equivalent to $M$. Hence for convenience we may assume the bottom face of~(\ref{cubePdiag}) is a point-set pushout. 
Define $Q$ as 
the point-set pushout of 
\(\namedright{P}{}{CP}\) 
with itself, so the top face strictly commutes. The strict commutativity of the left face, rear face and 
bottom face, and the fact that $Q$ is a point-set pushout, implies that there is a pushout map 
\(\namedright{Q}{\widehat{g}}{M}\) 
that makes the front and right faces of~(\ref{cubePdiag}) strictly commute. Therefore the 
cube~(\ref{cubePdiag}) strictly commutes, the bottom and top faces are pushouts, and the 
left and rear faces are homotopy pullbacks since $CP$ is contractible and there are rational homotopy fibrations 
\(\nameddright{P}{\widetilde{g}}{DV}{i}{V\times I}\) 
and 
\(\lnameddright{P}{\widetilde{g}}{DV}{i\circ e(h)}{V\times I}\). 
Hence, by Theorem~\ref{1cubthm}, the front and right faces in~(\ref{cubePdiag}) are also 
homotopy pullbacks. 

Finally, we draw consequences. Observe that as $CP$ is contractible, the pushout in 
the top face of~(\ref{cubePdiag}) implies that $Q\simeq\Sigma P$. Since the right face 
is a homotopy pullback, we obtain a diagram of rational homotopy fibrations 
 \[
\begin{gathered}
\xymatrix{
\Omega V \ar@{=}[d] \ar[r]  & 
\ast \ar[d] \ar[r]^{} &
V \times I\ar[d]^{j_1} \\
\Omega V \ar[r]^{\tau} &
\Sigma P \ar[r]^{\widehat{g}}&
M,
}
\end{gathered}
\]
implying that $\tau$ is null homotopic. Thus there is a rational homotopy equivalence $\Omega M\simeq_{\mathbb{Q}} \Omega V\times\Omega\Sigma P$, and Lemma \ref{DVdeclemma} then refines this 
to a rational homotopy equivalence  
$\Omega M \simeq_\mathbb{Q}\Omega V \times \Omega \Sigma^2 F$. 
\end{proof}


\section{The proof of Theorem \ref{openbookthm}}
\label{sec: proofthm1}

In this section, we prove Theorem \ref{openbookthm}.

\begin{proof}[Proof of Theorem \ref{openbookthm}]
We first prove case (a) when $i_{\ast}\circ e(h)_\ast: \pi_\ast(DV)\otimes \mathbb{Q}\rightarrow \pi_\ast(DV)\otimes \mathbb{Q}$ equals~$i_{\ast}$. 
By Theorem \ref{Mdecthm} there is a rational homotopy equivalence $\Omega M\simeq_{\mathbb{Q}}\Omega V\times\Omega \Sigma^2 F$, which implies that 
\[
\pi_\ast(M)\otimes \mathbb{Q}\cong (\pi_\ast(V)\otimes \mathbb{Q})\oplus (\pi_\ast(\Sigma^2 F)\otimes \mathbb{Q}).
\] 
Recall that $F$ is the homotopy fibre of the inclusion $\iota: \partial V\stackrel{}{\rightarrow} V$. 
By assumption, $V$ and $\partial V$ are nilpotent and the inclusion of the boundary 
\(\namedright{\partial V}{}{V}\) 
is $1$-connected. Thus $F$ is path-connected by Lemma \ref{FPconlemma}, and then~\cite[Chapter II, Proposition~2.13]{HMR} implies that $F$ is also nilpotent. 
Since $V$ is a smooth compact $(n-1)$-manifold, we have $H^{n-1}(V,\partial V;\mathbb{Z})\cong \mathbb{Z}$. This implies that $\iota$ is not a rational homotopy equivalence. In particular, $F$ is not rationally contractible. 
Therefore $\Sigma^2 F$ is not rationally contractible, in which case it is rationally a wedge of spheres. 
This implies that $\Sigma^{2} F$ is rationally hyperbolic unless it is a single sphere, in which case $F\simeq_{\mathbb{Q}} S^l$ for some $l$.
It follows that~$M$ is rationally elliptic if and only if $V$ is rationally elliptic and $F\simeq_{\mathbb{Q}} S^l$ for some $l$. In this case 
\[
\pi_\ast(M)\otimes \mathbb{Q}\cong (\pi_\ast(V)\otimes \mathbb{Q})\oplus (\pi_\ast(S^{l+2})\otimes \mathbb{Q}).
\] 
Otherwise $M$ is rationally hyperbolic, and this happens if and only if either $V$ is rationally hyperbolic or $F\not\simeq_{\mathbb{Q}} S^l$ for any $l$. This proves the theorem in the case (a) when $i_{\ast}\circ e(h)_\ast=i_{\ast}$. 

Next, we prove case (b). 
By assumption, there is an $m\in\mathbb{Z}^+$ such that $i_{\ast}\circ e(h)_\ast^m=i_{\ast}$. Notice that, by definition of $e(h)$, we have $e(h)^m=e(h^m)$. Therefore $i_{\ast}\circ e(h^m)_\ast=i_{\ast}$, so by the special case the theorem holds for the open book 
\[
M^\prime\cong (\partial V\times D^2)\cup_{\rm id} V_{h^m}.
\]
On the other hand, by definition of $V_{h}$ as a quotient space there is a quotient map
\[
q_h: V\times I \longrightarrow V_h
\]
such that $q_h(x, 1)=q_h(h(x),0)$. Denote $q_h(x, t)$ by $[x,t]$ for simplicity.
Define a map
\[
\widetilde{c}_m: V\times I\longrightarrow V_h,
\]
by $\widetilde{c}_m(x, \frac{i}{m}+t)= [h^i(x), mt]$ for any $0\leq i \leq m-1$ and $t\in [0,\frac{1}{m}]$. Since 
\[
\widetilde{c}_m(x, \frac{i}{m}+\frac{1}{m})=[h^i(x), 1]=[h^{i+1}(x), 0]=\widetilde{c}_m(x, \frac{i+1}{m}+0),
\]
$\widetilde{c}_m$ is well-defined and continuous. Further, since 
\[
\widetilde{c}_m(x, 1)=\widetilde{c}_m(x, \frac{m-1}{m}+\frac{1}{m})=[h^{m-1}(x),1]=[h^m(x), 0]=\widetilde{c}_m(h^m(x), 0),
\]
we see that $\widetilde{c}_m$ factors through $q_{h^m}$ to define a map
\[
c_m: V_{h^m}\longrightarrow V_h
\]
such that $c_m\circ q_{h^m}=\widetilde{c}_m$. By its construction, $c_m$ covers the standard $m$-sheeted covering $t_m: S^1\rightarrow S^1$ defined by $t_m(z)=z^m$ for any $z\in \mathbb{C}$, and by (\ref{Vhbundleeq}) there is a morphism of fibre bundles
 \begin{equation}
\begin{gathered} 
\label{cmtm} 
\xymatrix{
V\ar@{=}[d] \ar[r] &
V_{h^m} \ar[d]^{c_m} \ar[r]^{} &
S^1 \ar[d]^{t_m} \\
V \ar[r] &
V_{h} \ar[r] &
S^1.
}
\end{gathered}
\end{equation}
Further, since $h$ restricts to the identity map on $\partial V$, the $m$-sheeted covering $c_{m}$ restricts to $\rm{id}\times t_{m}$ on $\partial V\times S^{1}$. 

In order to apply rational homotopy theory to Diagram (\ref{cmtm}), we need to prove that the manifolds $V_h$ and $V_{h^m}$ are nilpotent spaces.
Indeed, for the fibre bundle (\ref{Vhbundleeq}) of the mapping torus $V_h$, the canonical action of $\pi_1(S^1)\cong \mathbb{Z}$ on the homotopy groups $\pi_\ast(V)$ of the fibre is determined by its restriction on a generator of $\pi_1(S^1)$, which is the monodromy action $h_\ast: \pi_\ast(V)\stackrel{}{\longrightarrow}\pi_\ast(V)$. By assumption in case (b), this action is nilpotent.
Moreover, the fibre bundle implies that $\pi_1(V_h)\cong \pi_1(S^1)\cong\mathbb{Z}$, $\pi_i(V_h)\cong \pi_i(V)$ for any $i\geq 2$, and up to homotopy equivalence $V$ is the universal covering of $V_h$. It follows that the action of the fundamental group $\pi_1(V_h)$ on the higher homotopy groups $\pi_i(V_h)$, $i\geq 2$ can be identified with the action of $\pi_1(S^1)$ on the homotopy groups $\pi_i(V)$, which is nilpotent by the previous discussion. Hence, the manifold $V_h$ is a nilpotent space. This implies that the $m$-sheeted covering $V_{h^m}$ is also a nilpotent space by \cite[Theorem 1.3]{Mis}. In particular, rationalizations of $V_h$ and $V_{h^m}$ exist. 

Now we draw consequences. First, since the map $t_{m}$ induces an isomorphism on rational homotopy groups, the five-lemma applied to the morphism of fibre bundles~(\ref{cmtm}) implies that $c_{m}$ also induces an isomorphism on rational homotopy groups. Whitehead's Theorem therefore implies that $c_{m}$ induces a rational homotopy equivalence. Second, the fact that $c_{m}$ restricts to ${\rm id}\times t_{m}$ on $\partial V\times S^{1}$ implies that there is a map 
\[
\Phi: (\partial V\times D^2)\cup_{\rm id} V_{h^m} \stackrel{}{\longrightarrow} (\partial V\times D^2)\cup_{\rm id} V_{h} 
\]
such that $\Phi$ restricts to $c_m$ on $V_{h^m}$ and to $({\rm id}\times t_m)$ on $(\partial V\times D^2)$. 
As all spaces involved are nilpotent, it follows that $\Phi$ is also a rational homotopy equivalence and $M\simeq_{\mathbb{Q}} M^\prime$. Therefore, as the theorem holds for $M'$, it also holds for $M$. This completes the proof of the theorem.
\end{proof}

\section{The necessity of the homotopy conditions on the monodromy} 
\label{sec:necessity} 

In this section we investigate the homotopy conditions on the monodromy in Theorem \ref{openbookthm} further. First is a comparison with work of Grove and Halperin and second is an application to Milnor's open books. 

In \cite{GH}, Grove and Halperin showed that if there are maps $\phi_i: X\rightarrow B_i$ ($i=0$, $1$) whose homotopy fibre are rationally spheres then the {\it double mapping cylinder}
\[
DCyl(X):= B_0\cup_{\phi_0} (X \times I) \cup_{\phi_1} B_1
\] 
is rationally elliptic. In our case, the open book $M\cong (\partial V\times D^2)\cup_{\rm id} V_h$ is a double mapping cylinder via the homeomorphism 
\[
M\cong (\partial V\times D^2) \cup_{({\rm id}\times j)} (\partial V\times S^1 \times I) \cup_{J} V_h, 
\] 
where $j: S^1\hookrightarrow D^2$ is the canonical inclusion and $J: \partial V\times S^1\hookrightarrow V_h$ is the restriction of the mapping torus to the boundary $\partial V$.
It is clear that the homotopy fibre of $({\rm id}\times j)$ is $S^1$. The morphism of bundles 
\[ 
\xymatrix{ 
  \partial V\ar[r]\ar[d] & \partial V\times S^{1}\ar[r]\ar[d]^{J} & S^{1}\ar@{=}[d] \\ 
  V\ar[r] & V_{h}\ar[r]^{\pi} & S^{1} 
} 
\]
implies that the left square is a pullback, and therefore the homotopy fibre of $J$ is the same 
as that of the inclusion $\partial V\hookrightarrow V$. When the latter fibre is rationally a sphere, Grove and Halperin \cite[Corollary 6.1]{GH} showed that $M$ is rationally elliptic if and only if $\partial V$ is. 
In particular, this is consistent with part (1) of Theorem \ref{openbookthm}. The classification result in Theorem~\ref{openbookthm} can therefore be thought of extending Grove and Halperin's work in the elliptic case. 

\begin{example}\label{triobsm+1ex}
Let $N\times S^{m+1}$ be the product of a closed manifold $N$ and the $(m+1)$-sphere with $m\geq 2$. Since $S^{m+1}\cong (S^{m-1}\times D^2)\cup_{(S^{m-1}\times S^1)} (D^m\times S^1)$, the product manifold $N\times S^{m+1}$ admits a canonical open book decomposition
\[
N\times S^{m+1}\cong (N\times S^{m-1}\times D^2)\cup_{{\rm id}} (N\times D^m\times S^1)
\]
with page $N\times D^m$ and trivial monodromy. Note that the homotopy fibre of the inclusion $N\times S^{m-1}\hookrightarrow N\times D^m$ is $S^{m-1}$. Then the conclusion of Theorem \ref{zMdecthm} reduces to the obvious homotopy equivalence
\[
\Omega(N\times S^{m+1})\simeq  \Omega (N\times D^m)\times \Omega \Sigma^2 S^{m-1}. 
\] 
Further, the conclusion of Theorem \ref{openbookthm} reduces to stating that $N\times S^{m+1}$ is rationally elliptic if and only if $N\times D^m$ is, which automatically holds as $S^{m+1}$ is rationally elliptic.
\end{example}

\begin{example}
Following Example \ref{triobsm+1ex}, we can try to construct examples with nontrivial monodromies, for which we need to choose nontrivial homotopy classes of elements in the diffeomorphism group ${\rm Diff}_{\partial} D^{m}$, the group of self-diffeomorphisms of $D^m$ which fix a neighborhood of $\partial D^m$ pointwise. It is known that ${\rm Diff}_{\partial} D^{m}$ is contractible for $m=1$, $2$ and $3$: the case $m=1$ is easy, the case $m=2$ was proved by Smale \cite{S}, and the case $m=3$, known as the Smale conjecture, was proved by Hatcher \cite{Hat}. In particular, when $m=2$ or $3$ we cannot construct an open book with page $D^m$ and nontrivial monodromy. 

However, for $m\geq 4$ there could be such constructions. Indeed, for each $m\geq 5$ Cerf \cite{Ce} proved that
\[
\pi_0({\rm Diff}_{\partial} D^{m})\cong \Theta_{m+1}
\]
where $\Theta_{m+1}$ denotes the group of oriented homotopy $(m + 1)$-spheres up to
oriented diffeomorphism. As Kervaire and Milnor \cite{KM} showed that $\Theta_{m+1}$ is always finite, the monodromy $h$ of the open book construction
\[
\mathcal{S}_h:=(S^{m-1}\times D^2)\cup_{\rm id} D^m_h
\]
is of finite order up to isotopy. Since the page $D^m$ is contractible, the induced homomorphism $i_\ast: \pi_\ast(DV)\longrightarrow \pi_\ast(V\times I)$ is trivial, and the homotopy fibre of the inclusion $S^{m-1}\hookrightarrow D^m$ is $S^{m-1}$, Theorem \ref{openbookthm} applies to show that $\mathcal{S}_h$ is rationally elliptic and 
\[
\pi_\ast(\mathcal{S}_h)\otimes \mathbb{Q}\cong \pi_\ast(S^{m+1})\otimes \mathbb{Q}.
\]
In particular, $\mathcal{S}_h$ is rationally a sphere.

For $m=4$, the famous recent work of Watanabe \cite{Wa} disproved the conjecture that ${\rm Diff}_{\partial} D^{4}$ is contractible. However, as in \cite[Remark 1.2]{Wa}, information about $\pi_0({\rm Diff}_{\partial} D^{4})$ is still unknown. Nevertheless, as above, any class in $\pi_0({\rm Diff}_{\partial} D^{4})$ will still give an open book structure on a corresponding rational sphere.
\end{example}

Theorem \ref{openbookthm} also has an interesting application in the hyperbolic case. Consider Milnor's open book decompositions of spheres \cite{Mil}. Let $f: \mathbb{C}^{n+1}\rightarrow \mathbb{C}$ be a non-constant complex polynomial in $z=(z_1,\ldots, z_{n+1})$, and $\mathcal{Z}=f^{-1}(0)$ be the algebraic set of zeros. Suppose $z_{0}\in \mathcal{Z}$ is an isolated critical point of $f$ and there is a diffeomorphism $S_{\epsilon}\cong S^{2n+1}$, where $S_{\epsilon}$ is a sphere of radius $\epsilon$ centered at $z_{0}$ for some sufficiently small $\epsilon$. Then there is a compact $2n$-manifold $V$ with boundary and an abstract open book decomposition of $S^{2n+1}$ via diffeomorphisms 
\[
S^{2n+1}\cong S_{\epsilon}\cong (\partial V\times D^2)\cup_{\rm id} V_h,
\]
where the binding is $\partial V\cong \mathcal{Z}\cap S_{\epsilon}$ and the interior $(V-\partial V)$ of the page $V$ is diffeomorphic to the fibre of the {\it Milnor fibration}
\[
\phi: S_{\epsilon}-\mathcal{Z}\longrightarrow S^1
\] 
defined by $\phi(z)=\frac{f(z)}{||f(z)||}$. Milnor showed that there is a homotopy equivalence 
\begin{equation}\label{VmuSneq}
V\simeq \bigvee_{i=1}^{\mu} S^{n},
\end{equation}
where the positive integer $\mu$ is the multiplicity of $z_0$ as solution to the system of polynomial equations $\{\partial f/ \partial z_j=0\}_{j=1}^{n+1}$.
\begin{proposition}\label{milcoro}
Let $h$ be the monodromy of the open book $S^{2n+1}$ determined by an isolated critical point $z_0$ of a complex polynomial $f: \mathbb{C}^{n+1}\rightarrow \mathbb{C}$ with $n\geq 3$. If the multiplicity of $z_0$ as solution to the system of polynomial equations $\{\partial f/ \partial z_j=0\}_{j=1}^{n+1}$ is greater than $1$, then either:
\begin{itemize}
\item there is no integer $m$ such that $h^m$ is rationally homotopic to the identity map, or 
\item the monodromy $h$ acts non-nilpotently on the homology groups $H_{\ast}(V;\mathbb{Z})$.
\end{itemize}

\end{proposition}
\begin{proof}
To obtain a contradiction assume that $h^m$ is rationally homotopic to the identity map and the monodromy $h$ acts nilpotently on the homology groups $H_{\ast}(V;\mathbb{Z})$. By a similar argument to that in the proof of Theorem \ref{openbookthm}, the latter condition means that $\pi_1(V_h)$ acts nilpotently on the homology groups $H_\ast(V)$ of the universal cover. It follows that $V_h$ is a nilpotent space by \cite[Chapter II, Remark~2.19]{HMR}, which implies that $h$ acts nilpotently on the homotopy groups $\pi_\ast(V)$.
Further, as $n\geq 3$, the binding $\partial V$ is simply-connected by \cite[Proposition 5.2]{Mil}.
For the page, by (\ref{VmuSneq}) $V\simeq\bigvee_{i=1}^{\mu} S^n$ and it is rationally hyperbolic as $\mu\geq 2$ by hypothesis. Therefore Theorem \ref{openbookthm} can be applied and it implies that the associated open book $S^{2n+1}=(\partial V\times D^{2})\cup_{{\rm id}} V_{h}$ is rationally hyperbolic. This contradicts the fact that spheres are rationally elliptic, and hence the proposition follows. 
\end{proof}

Notice too that for the open book in Proposition \ref{milcoro} with $\mu \geq 2$, the homotopy fibre of $\partial V\hookrightarrow V$ cannot be a sphere rationally. Otherwise, by the result of Grove and Halperin \cite[Corollary 6.1]{GH} $\partial V$ is rationally elliptic as $S^{2n+1}$ is, and then so is $V$. However, this contradicts (\ref{VmuSneq}) when $\mu \geq 2$.

We end this section with two examples.
\begin{example}\label{Beg}
Consider the Brieskorn variety given by the complex polynomial
\[
f(z_1,\ldots, z_{n+1})=(z_1)^{a_1}+\cdots +(z_n)^{a_n}+(z_{n+1})^{a_{n+1}}
\]
with $n\geq 3$ and $a_1, \ldots, a_{n+1}\geq 2$.
The origin $z_0=0$ is the only critical point of $f$.
Brieskorn and Pham \cite[Theorem 9.1]{Mil} showed that the multiplicity of $0$ is
\[
\mu= \mathop{\prod}\limits_{j=1}^{n+1} (a_i-1).
\]
Hence, for the associated open book decomposition of $S^{2n+1}$, when ${\rm max}(a_1,\ldots, a_{n+1})\geq 3$, by Proposition \ref{milcoro} either the monodromy $h$ is of infinite order, or $h$ acts non-nilpotently on the homology groups $H_{\ast}(V;\mathbb{Z})$. 
\end{example}

\begin{example}\label{Keg}
Suppose that $n$ is even and each $a_i=2$, $1\leq i\leq n+1$ in Example \ref{Beg}. 
Then for the associated open book decomposition of $S^{2n+1}$, by \cite[Example 2]{KK} the monodromy $h$ is of order at most $8$, and its page $V$ is diffeomorphic to the disk bundle of the tangent bundle of $S^n$. 

Consider the boundary connected sum $V\# V$ with its self diffeomorphism $h\# h$. It determines an open book diffeomorphic to $((\partial V\# \partial V)\times D^2)\cup_{\rm id} (V\# V)_{h\# h}$.
Recall for an open book of dimension $2n+1$ with page $X$ and monodromy $\varphi$, its {\it variation} of the monodromy
\[
{\rm Var}(\varphi): H_n (X, \partial X)\stackrel{}{\longrightarrow} H_n(X)
\]
is defined by ${\rm Var}(\varphi)[z]=[\varphi(z)-z]$ for any relative $n$-cycle $z$ of $(X, \partial X)$. A classical result of Kauffman \cite[Proposition 3.3]{Kau} shows that an open book is a homotopy sphere if and only its variation is an isomorphism. In particular, ${\rm Var}(h)$ is an isomorphism for the open book decomposition of $S^{2n+1}$ at the start of the example. Since $V\# V\simeq S^{n}\vee S^{n}$, it is easy to check that ${\rm Var}(h\# h)$ is also an isomorphism, and hence we obtain an open book decomposition of a homotopy sphere $\Sigma^{2n+1}$
\[
\Sigma^{2n+1}\cong ((\partial V\# \partial V)\times D^2)\cup_{\rm id} (V\# V)_{h\# h}.
\]

Note that its monodromy $h\# h$ is of order at most $8$ as $h$ is, its page $V\# V\simeq S^{n}\vee S^{n}$ is rationally hyperbolic, and $\Sigma^{2n+1}\simeq S^{2n+1}$ is rational elliptic. Therefore, Theorem \ref{openbookthm} implies that the monodromy $h\# h$ acts non-nilpotently on the homology groups $H_{\ast}(V\# V;\mathbb{Z})$. 
\end{example}

\appendix
\section{Mather's cube theorems}
\label{AppendixA}

In this appendix, we describe two cube theorems due to Mather \cite{M}. The statements 
are special cases of those in~\cite{M} to avoid technicalities with compatible homotopies, 
as will be explained. The first statement is from~\cite[Theorem 5.10.7]{MV}. 

\begin{theorem}[Mather's first cube theorem]\label{1cubthm}
Suppose that there is a strictly commutative diagram
\begin{gather}
\begin{aligned}
\xymatrixcolsep{1.5pc}
\xymatrixrowsep{1.5pc}
\xymatrix{
 & A^\prime \ar[dl]  \ar[rr]  \ar[dd]^<<<<{f_1}|!{[d];[d]}\hole && 
 B^\prime \ar[dl]  \ar[dd]^{f_2}  \\
 C^\prime \ar[dd]^{f_3}  \ar[rr]  &&
 D^\prime \ar[dd]^<<<<{f_4} \\
  & A \ar[dl]  \ar[rr]|!{[r];[r]}\hole  && 
 B\ar[dl]    \\
 C \ar[rr]  &&
 D 
}
\end{aligned}
\label{cubethmdiag}
\end{gather}
in which 
\begin{itemize}
\item the left and rear faces are homotopy pullbacks;
\item the top and bottom faces are homotopy pushouts; and
\item the map $f_4$ is the naturally induced map 
\begin{equation}\label{indf4eq}
{\rm ho colim}(C^\prime\leftarrow A^\prime \rightarrow B^\prime)\longrightarrow {\rm ho colim}(C\leftarrow A\rightarrow B).
\end{equation}
\end{itemize}
Then the front and the right faces are homotopy pullbacks.~$\qqed$  
\end{theorem}

The original statement of the first cube theorem~\cite[Theorem 18]{M} assumes that the cube 
homotopy commutes and satisfies compatibilities among the homotopies. The hypothesis of strict 
commutativity for the cube in Theorem~\ref{1cubthm} lets one bypass the compatibilties. 
The second statement follows~\cite[Lemma 3.1]{PT}.

\begin{theorem}[Mather's second cube theorem]\label{2cubthm} 
Suppose that there is a homotopy pushout 
\[\diagram 
     A\rto\dto &  B\dto \\ 
     C\rto & D 
  \enddiagram\] 
and a homotopy fibration 
\(\nameddright{D'}{}{D}{h}{Z}\). 
Composing all the maps in the homotopy pushout with~$h$ and taking homotopy fibres over 
the common base $Z$ gives a homotopy commutative cube
\begin{gather}
\begin{aligned}
\xymatrixcolsep{1.5pc}
\xymatrixrowsep{1.5pc}
\xymatrix{
 & A^\prime \ar[dl]  \ar[rr]  \ar[dd]|!{[d];[d]}\hole && 
 B^\prime \ar[dl]  \ar[dd]  \\
 C^\prime \ar[dd]  \ar[rr]  &&
 D^\prime \ar[dd] \\
  & A \ar[dl]  \ar[rr]|!{[r];[r]}\hole  && 
 B\ar[dl]    \\
 C \ar[rr]  &&
 D 
 }
\end{aligned}
\end{gather}
that defines $A'$, $B'$ and $C'$, and for which the four sides are all homotopy pullbacks.  
Then the top face of the cube is a homotopy pushout.~$\qqed$ 
\end{theorem} 

The original statement of the second cube theorem~\cite[Theorem 25]{M} assumes that 
there is a cube that homotopy commutes in which the bottom face is a homotopy pushout, 
the four sides are homotopy pullbacks, and there are compatibilities among the homotopies, 
and concludes that the top face is a homotopy pushout. The hypothesis that the homotopy  
commutativity of the cube is due to it being induced by taking fibres over a map 
\(\namedright{D}{h}{Z}\)  
lets one bypass the compatibilities.

\bibliographystyle{amsalpha}

\end{document}